\documentclass[a4paper,12pt,reqno]{article}
\usepackage[T1]{fontenc}
\usepackage{authblk}
\usepackage{epsfig,graphicx,subfigure}
\usepackage{amsthm,amsmath,amssymb,amsfonts}
\usepackage{color,xcolor}
\usepackage[all]{xy}
\usepackage[top=40mm, bottom=40mm, left=40mm, right=35mm]{geometry}
\usepackage{hyperref}

\newtheorem{theorem}{Theorem}[section]

\newtheorem{corollary}[theorem]{Corollary}
\theoremstyle{definition}

\newtheorem{example}[theorem]{Example}

\newtheorem{remark}[theorem]{Remark}
\numberwithin{equation}{section}

\begin{document}
\setcounter{page}{1}
\title{\vspace{-1.5cm}
\vspace{.5cm}
\vspace{.7cm}
{\large{\bf  Symmetric bi-derivations on certain Banach algebras} }}
 \date{}
\author{{\small \vspace{-2mm}  M. Eisaei$^1$ and Gh. R. Moghimi $^1$\footnote{Corresponding author} }}
\affil{\small{\vspace{-4mm}  $^1$ Department of Mathematics,
Payame Noor University (PNU), Tehran 19395-4697, Iran }}
\affil{\small{\vspace{-4mm}  mojdehessaei59@student.pnu.ac.ir }}
\affil{\small{\vspace{-4mm} moghimimath@pnu.ac.ir}}
\maketitle
\hrule
\begin{abstract}
\noindent
Let $A$ be a Banach algebra with a right identity $u$ such that $uA$ is commutative and semisimple. In this paper, we investigate symmetric bi-derivations of $A$ and detremine their range. We also study symmetric bi-derivations of $A$ with their $k$-centralizing trace. Finally, we prove every symmetric Jordan bi-derivation of $A$ is a symmetric bi-derivation.
 \end{abstract}

 \noindent \textbf{Keywords}: Bi-derivation, $k$-centralizing, $k$-skew centralizing.\\
{\textbf{2020 MSC}}:  47B47, 16W25.
\\
\hrule
\vspace{0.5 cm}
\baselineskip=.8cm
\section{Introduction}

Let $A$ be a Banach algebra.
A mapping $D: A \times A \rightarrow A$ is called a \textit{symmetric bi-linear}  mapping  if $D(a,b)=D(b,a)$,
\begin{equation*}
D( \alpha a, b)= \alpha D(a,b) \quad \text{and} \quad D(a+b,c)=D(a,c)+ D(b,c).
\end{equation*}
for all $a,b,c \in A$ and $\alpha \in \Bbb{C}$. Also,
$D$ is called a \textit{symmetric Jordan bi-derivation} if
$D$ is symmetric bi-linear and
\begin{equation*}
D(a^2,b)= D(a,b) a + a D(a,b)
\end{equation*}
for all $a,b \in A$.
Furthermore, if for every $a,b,c \in A$, we have
\begin{equation*}
D(ab, c)= D(a,c) b + a  D(b,c), 
\end{equation*}
then $D$ is called a \textit{symmetric bi-derivation}.
Clearly, every Jordan bi-derivation is a bi-derivation.
But, the converse is not true, in general.
The mapping $f: A \rightarrow A$ defined by
\begin{equation*}
f(a)=D(a,a)
\end{equation*}
is called the \textit{trace of } $D$.

Let $T: A \rightarrow A$ be a mapping and $k \in \Bbb{N}$. Then $T$ is called 
\textit{$k$-centralizing \emph{(}resp. $k$-skew centralizing\emph{)}}  if for every $a \in A$,
\begin{equation*}
[T(a), a^k]:= T(a) a^k - a^k T(a) \,\,\, (\text{resp. } \langle T(a), a^k \rangle= T(a) a^k + a^k T(a) )
\end{equation*}
is an element of $Z(A)$, the set of all $a\in A$ such that $ax=xa$ for all $x\in A$. In particular, if for every $a \in A$
\begin{equation*}
[T(a), a^k]=0, \,\,\, (\text{resp.} \langle T(a), a^k \rangle =0).
\end{equation*}
then $T$ is called \emph{$k$-commuting} (resp. \emph{$k$-skew commuting}). 

Maksa \cite{m1, m2} introduced and studied symmetric bi-derivations.
Some authors continued this investigation \cite{am, d, p, v1, v2, w}.
For example, vukman \cite{v2} showed that if
$D: R \times R \rightarrow R$ is a symmetric bi-derivation on a non commutative prime ring of characteristic note two and there $R$ such that for every $r \in R$,
\begin{equation*}
[[f(r),r],r] \in Z(R),
\end{equation*}
then $D$ is zero.

In this paper, we always assume that $A$ is a Banach algebra with a right identity $u$ such that $uA$ is commutative and semisimple. We prove that every symmetric bi-derivation of $A$ maps $A$ into $\hbox{ran}(A)$. We also show that if the trace of a symmetric bi-derivation of $A$ is $k$-centralizing, then the symmetric bi-derivation is zero. Finally, we establish that every symmetric Jordan bi-derivation of $A$ is a symmetric bi-derivation.
\section{Main Results}
 Before we give the first our result, let us recall that the right annihilator of $A$ is denoted by $\hbox{ran}(A)$ and it is the set of all $r\in A$ such that $ar=0$ for all $a\in A$. Note that $\hbox{ran}(A)$ is a subset of the radical of $A$.

\begin{theorem} \label{m1}
Let $D: A \times A \rightarrow A$ be a symmetric bi-derivation. Then the range of $D$ is contained into $ \emph{ran}(A)$.
\end{theorem}
\begin{proof}
First, note that  the range of $D$ is equal to
$\cup_{m \in A} \phi_m(A)$,
where $$\phi_m(n)= D(m,n)$$ for all
$n \in A$.
But $\phi_m$ is a derivation on $A$ and so
$\phi_m(A) \subseteq  \hbox{ran}(A)$; see \cite{mms}.
This proves the theorem.
\end{proof}
\begin{remark}
Let $\mathcal{A}$ be a Banach algebra such that $\frac{\mathcal{A}}{\hbox{ran}(A)}$ is commutative and semisimple.
Then a similar argument to the proof of Theorem 2.1 shows that the range of any symmetric bi-derivation $D: \mathcal{A} \times \mathcal{A} \rightarrow \mathcal{A}$ is contained into
$ \hbox{ran}(\mathcal{A})$.
\end{remark}
\begin{theorem} \label{m2}
Let $D: A \times A \rightarrow A$ be a symmetric bi-derivation and $f$ be the trace of $D$.
Then the following assertions are equivalent.

\emph{(a)} There exists $k \in \Bbb{N}$ such that $f$ is $k$-commuting.

\emph{(b)} There exists $k \in \Bbb{N}$ such that $f$ is $k$-centralizing.

\emph{(c)} There exists $k \in \Bbb{N}$ such that $f$ is $k$-skew commuting.

\emph{(d)} There exists $k \in \Bbb{N}$ such that $f$ is $k$-skew centralizing.

\emph{(e)} $D=0$.
\end{theorem}
\begin{proof}
Let $m \in A$ and $k \in \Bbb{N}$.
It follows from Theorem 2.1 that
\begin{eqnarray*}
[D(m,m), m^k]&=& D(m,m) m^k - m^k D(m,m) \\
&=& D(m,m) m^k + m^k D(m,m) \\
&=& \langle D(m,m), m^k \rangle.
\end{eqnarray*}
Thus
\begin{equation*}
[f(m), m^k] = \langle f(m), m^k \rangle.
\end{equation*}
So the statements (a)-(d) are equivalent.
Assume now that $f$ is $k$-commuting.
Then
\begin{eqnarray} \label{po1}
f(m) m^k &=& D(m,m) m^k  \nonumber \\
&=& D(m,m) m^k - m^k D(m,m) \nonumber  \\
&=& [D(m,m), m^k]\\
&=& [f(m), m^k]\nonumber\\
&=&0\nonumber
\end{eqnarray}
for all $m \in A$.
Hence $f(u)=0$. We also have
\begin{eqnarray*}
f(m+u)&=& D(m+u, m+u) \nonumber \\
&=& D(m,m) +2 D(m,u)+ D(u,u) \nonumber  \\
&=& f(m) + 2 D(m,u)
\end{eqnarray*}
for all $m \in A$. This together with (2.1) shows that
\begin{eqnarray*}
0&=& f(m+u) (m+u)^k \\
&=& (f(m)+ 2 D(m,u)) (m+u)^k \\
&=& (f(m)+ 2 D(m,u)) \bigg ( \sum_{j=0}^{k-1} \binom{k-1}{j}m^{k-j} + \sum_{j=0}^{n-1} \binom{n-1}{j} u m^{k-j-1} \bigg)
\end{eqnarray*}
for all $m \in A$. This implies that
\begin{eqnarray} \label{po3}
 \sum_{j=1}^{k} \binom{k}{j} f(m) m^{k-j} +2 \sum_{j=0}^{k} \binom{k}{j} D(m,u) m^{k-j}  =0.
\end{eqnarray}
Put
\begin{eqnarray*}
E_1(m)=  \sum_{ \tiny{\begin{array}{c}
j \, even  \\
 j=2
\end{array}}}^{k} \binom{k}{j} f(m) m^{k-j}, \quad
E_2(m)= 2  \sum_{ \tiny{\begin{array}{c}
j \, even  \\
 j=0
\end{array}}}^{k} \binom{k}{j} D(m,u) m^{k-j},
\end{eqnarray*}
\begin{eqnarray*}
O_1(m)=  \sum_{ \tiny{\begin{array}{c}
j \, odd  \\
 j=1
\end{array}}}^{k} \binom{k}{j} f(m) m^{k-j}, \quad
O_2(m)= 2  \sum_{ \tiny{\begin{array}{c}
j \, odd  \\
 j=1
\end{array}}} \binom{k}{j} D(m,u) m^{k-j}.
\end{eqnarray*}
It follows from (2.2) that
\begin{equation} \label{op4}
E_1(m) + E_2(m) + O_1(m)+ O_2(m)=0
\end{equation}
for all $m \in A$.
If we replace $m$ by $-m$ in (2.3), then
\begin{equation*}
E_1(m) - E_2(m) - O_1(m)+ O_2(m)=0
\end{equation*}
Hence
\begin{equation} \label{op5}
E_1(m)+O_2(m)=0
\end{equation}
and so
\begin{equation} \label{op6}
O_1(m)+ E_2(m)=0.
\end{equation}
Now, let $k$ be even. Then by (2.5), we have
\begin{eqnarray*}
0 &=& O_1(m) + E_2(m) \\
&=&   \sum_{ \tiny{\begin{array}{c}
j \, odd  \\
 j=1
\end{array}}}^{k-1} \binom{k}{j} f(m) m^{k-j} +
2 \sum_{ \tiny{\begin{array}{c}
j \, even   \\
 j=0
\end{array}}}^{k-2}
\binom{k}{j}  D(m,u) m^{k-j} + 2 D(m,u).
\end{eqnarray*}
Thus for every $r \in \hbox{ran} (A)$, we obtain
\begin{equation} \label{op7}
D(r,u)=0,
\end{equation}
because $$f(r) r^{k-j} = D(r,u) r^{k-j}=0.$$
Since $m-um \in \hbox{ran}(A)$, it follows that
\begin{eqnarray*}
0&=& D(m-um,u) \\
&=& D(m,u) - D(u,u) m- uD(m,u) \\
&=& D(m,u) - f(u) m \\
&=& D(m,u)
\end{eqnarray*}
for all $m \in A$.
Hence $O_2=0$ and thus
\begin{equation} \label{op8}
\sum_{ \tiny{\begin{array}{c}
j \, even   \\
 j=2
\end{array}}}^{k}
\binom{k}{j}  f(m) m^{k-j}=0.
\end{equation}
If $i$ is even and $2 \leq i \leq k-2$, then by
(2.7), we have
\begin{eqnarray*}
0&=&
\sum_{ \tiny{\begin{array}{c}
j \, even   \\
 j=2
\end{array}}}^{k}
\binom{k}{j}  f(m) m^{k+i-j} \\
&=& \sum_{ \tiny{\begin{array}{c}
j \, even   \\
 j=2
\end{array}}}^{k}
\binom{k}{j}  f(m) m^{k} m^{i-j} +
 \sum_{ \tiny{\begin{array}{c}
j \, even   \\
 j=i+2
\end{array}}}^{k}
\binom{k}{j}  f(m) m^{k+i-j}\\
&=& \sum_{ \tiny{\begin{array}{c}
j \, even   \\
 j=i+2
\end{array}}}^{k}
\binom{k}{j}  f(m) m^{k+i-j}.
\end{eqnarray*}
For $i=k-2$, we get
$f(m) m^{k-2} =0$.
Hence
\begin{equation*}
\sum_{ \tiny{\begin{array}{c}
j \, even   \\
 j=4
\end{array}}}^{k}
\binom{k}{j}  f(m) m^{k-j}=0\quad\hbox{ and} \quad
\sum_{ \tiny{\begin{array}{c}
j \, even   \\
 j=i+2
\end{array}}}^{k}
\binom{k}{j}  f(m) m^{k+i-j}=0.
\end{equation*}
Continuing this procedure, we have $f=0$ and therefore $D=0$.
Similarly, if  $k$ is odd, one can prove that $D=0$.
\end{proof}
\begin{theorem} \label{m4}
Let $D: A \times A \rightarrow A$ be a symmetric Jordan bi-derivation.
Then $D$ is a symmetric bi-derivation.
\end{theorem}
\begin{proof}
Let $m \in A$ and $\phi_m$ be the function defined as the proof of Theorem 2.1. Then for every $n \in A$, we have
\begin{eqnarray*}
\phi_m(n^2) &=& D(m,n^2) \\
&=& D(m,n) n + n D(m,n)\\
& =& \phi_m(n) n + n \phi_m(n).
\end{eqnarray*}
Thus $\phi_m$ is a Jordan derivation on $A$ for all $m \in A$.
Hence, $\phi_m$ is a derivation on $A$, see Theorem 2.2 of \cite{mms}. Therefore,
\begin{eqnarray*}
D(m k, n) &=& \phi_n(mk) \\
&=& \phi_n(m) k + m \phi_n (k) \\
&=& D(m,n)k + m D(k,n)
\end{eqnarray*}
for all $m,n,k \in A$.
That is, $D$ is a symmetric bi-derivation.
\end{proof}

As an immediate consequence of Theorems 2.1 and 2.4 we have the following result.

\begin{corollary} Let ${\cal A}$ be a unital commutative semisimple Banach algebra. Then the zero map is the only symmetric bi-derivation on ${\cal A}$.
\end{corollary}

We finsh the paper with the following example.

\begin{example} Let $\omega$ be a weight function on $[0, \infty)$. 

(i) Let $M(\omega)$ be the Banach algebra of all complex regular Borel measures $ \mu $ on $[0, \infty)$ such that $ \omega \mu  \in M( [0, \infty)) $, the Banach algebra of all complex regular Borel measures on $ [0, \infty)$. It is well-known that $M(\omega)$ is a unital commutative semisimple Banach algebra; see \cite{dl, r0}. Hence  the only symmetric bi-derivation on $M(\omega)$ is zero.

(ii) Let $L_0^\infty(\omega)$ be the Banach space of all Lebesgue measure functions $f$ on $[0, \infty)$ such that
\begin{equation*}
\lim_{x \rightarrow \infty} \text{ess sup}
 \bigg \{ \frac{f(y) \chi_{(x, \infty)} (y)}{w(y)}: y \geq 0 \bigg\}=0,
\end{equation*}
where $\chi_{(x, \infty)}$ is the characteristic function of $(x, \infty)$ on $[0, \infty)$. Then $L_0^\infty(\omega)^*$ is a Banach algebra \cite{mnr}. Also, $L_0^\infty(\omega)^*$ has a right identity, say $u$, and $uL_0^\infty(\omega)^*$ is isometrically isomorphic to $M(\omega)$; for an extensive study of this Banach algebra see \cite{am1, lp, mnr, mm, mn}. One can prove that the radical of $L_0^\infty(\omega)^*$ is equal to $\hbox{ran}(L_0^\infty(\omega)^*)$; see \cite{ms}. Therefore, every symmetric Jordan bi-derivation of $L_0^\infty(\omega)^*$ is a symmetric bi-derivation and its range is into the radical of $L_0^\infty(\omega)^*$. Also, there is no non-zero symmetric bi-derivation with $k$-cerntralizing trace.
\end{example}

\end{document}